\documentclass{siamltex}

\usepackage{algorithmic,algorithm}
\usepackage[leqno]{amsmath}
\usepackage{amsfonts, amssymb, amscd, mathrsfs}

\usepackage{color}

\definecolor{CKColor}{RGB}{255,0,0}

\newtheorem{assumption}{Assumption}[section]

\newcounter{remark-count}
\newenvironment{remark}[1][Remark]{ \refstepcounter{remark-count} \begin{trivlist} \item[\hskip \labelsep{ \hskip 0.5cm  \itshape #1 \thesection.\arabic{remark-count}.}]   }{\end{trivlist}}

\numberwithin{remark-count}{section}
\numberwithin{algorithm}{section}

\title{Comparative Convergence Analysis of Nonlinear AMLI-cycle Multigrid \thanks{This work was performed under the auspices of the U.S. Department of Energy by Lawrence Livermore National Laboratory under Contract DE-AC52-07NA27344.}}

\author{Xiaozhe Hu \thanks{Department of Mathematics, The
  Pennsylvania State University, University Park, PA 16802, U.S.A. (\tt hu\_x@math.psu.edu)}
        \and Panayot S. Vassilevski \thanks{Center for Applied Scientific Computing,
             Lawrence Livermore National Laboratory,
             P.O. Box 808, L-560,
             Livermore, CA 94550, U.S.A. (\tt panayot@llnl.gov). }
        \and Jinchao Xu \thanks{Department of Mathematics, The
  Pennsylvania State University, University Park, PA 16802, U.S.A. ({\tt xu@math.psu.edu}). The research of this author was supported in part by NSF Grant DMS-0749202 and DMS-0915153, DOE Grant DE-SC0006903, DOE Subcontracts B574178 and B591217, and NSFC Grant 91130011/A0117}
        }

\date{March 25, 2011--beginning; Today is \today}

\begin{document}

\maketitle

\begin{abstract}
The main purpose of this paper is to provide a comprehensive convergence analysis of nonlinear AMLI-cycle multigrid method for symmetric positive definite problems. Based on classical assumptions for approximation and smoothing properties, we show that the nonlinear AMLI-cycle MG method is uniformly convergent. Furthermore, under only the assumption that the smoother is convergent, we show that the nonlinear AMLI-cycle method is always better (or not worse) than the respective V-cycle MG method. Finally, numerical experiments are presented to illustrate the theoretical results.
\end{abstract}


\pagestyle{myheadings}
\thispagestyle{plain}
\markboth{X. Hu, P.~S. Vassilevski, and J.Xu}{Convergence Analysis of Nonlinear AMLI-cycle MG}

\section{Introduction} \label{sec:intro}
In this paper, we consider the large-scale sparse linear system of equations
\begin{equation} \label{eqn:model}
Au=f,
\end{equation}
where~$A$~is a symmetric positive definite (SPD) operator on a finite-dimensional vector space $V$. The development of efficient and practical solvers for large-scale sparse linear systems of equations arising from discretizations of partial differential equations (PDEs) is an important task in scientific and engineering computing. We consider an iterative solution for equation~\eqref{eqn:model} using a multigrid (MG) method.  Efficient, scalable, and often computationally optimal, MG methods have been used successfully in practical applications.  In fact, there is extensive literature on MG methods; see \cite{Hackbusch.W1985, Xu.J1989, Xu.J1992, Bramble.J1993, Briggs.W;McCormick.S2000, Trottenberg.U;Oosterlee.C;Schuller.A2001,Xu.J;Zikatanov.L2002, Vassilevski.P2008}, and references therein for details.  MG methods, especially their algebraic variants, algebraic multigrid (AMG) methods, are being increasingly used in practice. 
Originating in \cite{Brandt.A;McCormick.S;Ruge.J1985}, AMG method gained some popularity after \cite{Ruge.J;Stuben.K1987} appeared.  And more recently, researchers have further extended these methods and developed them in various directions (\cite{Vanek.P;Mandel.J;Brezina.M1996,Brezina.M;Cleary.A;Falgout.R;Henson.V;Jones.J;Manteuffel.T;McCormick.S;Ruge.J2001,Chartier.T;Falgout.R;Henson.V;Jones.J;Manteuffel.T;McCormick.S;Ruge.J;Vassilevski.P2003, Xu.J;Zikatanov.L2004, Brezina.M;Falgout.R;MacLachlan.S;Manteuffel.T;McCormick.S;Ruge.J2005,Lashuk.I;Vassilevski.P2008}, etc.). 

In order to improve the robustness of (A)MG methods, we usually use them as preconditioners in Krylov subspace iterative methods, such as the conjugate gradient (CG) method 
in the case when  $A$ is SPD.

The performance and efficiency of MG methods may degenerate when the physical and geometric properties of the problems become more and more complicated. Generally speaking, if the convergence factor of the two-grid method is too large, the fast convergence of the MG methods, which is expected to be independent of the levels, cannot be guaranteed with either the standard V-cycle or even with the standard W-cycle.  The multilevel cycle, which uses the best polynomial approximation of degree $n$ to define the coarse-level solver, was originally introduced in \cite{Axelsson.O;Vassilevski.P1989a,Axelsson.O;Vassilevski.P1990,Vassilevski.P1992b} and applied to the hierarchical basis MG method. This cycle, usually referred to as the algebraic multilevel iteration (AMLI) cycle, is designed to provide an optimal condition number,  if the degree $n$ of the polynomial is sufficiently large, 
under the assumption that the V-cycle MG method has bounded condition number that depends entirely on the difference of levels.
This assumption (on the bounded level length V-cycle convergence) is feasible for certain second-order elliptic PDEs without additional  assumptions in regard to PDE regularity. 

More recently, thanks to the introduction of the nonlinear (variable-step/flexible) preconditioning method and the analysis of it in \cite{Axelsson.O;Vassilevski.P1991a} (see also \cite{Golub.G;Ye.Q1999a,Notay.Y2000,Saad.Y2003}, etc.), nonlinear multilevel preconditioners were proposed and additive version of them analyzed in \cite{Axelsson.O;Vassilevski.P1994}. 
Furthermore, the multiplicative version was investigated in \cite{Kraus.J2002}. In these nonlinear multilevel preconditioners, $n$ steps of a preconditioned CG iterative method replace the 
best polynomial approximation and are performed to define the coarse-level solvers. The condition number is optimal for  properly chosen $n>1$. 
The same idea can be used to define the MG cycles, as shown for the first time in~\cite{Vassilevski.P2008}. The resulting nonlinear AMLI-cycle MG was analyzed in \cite{Notay.Y;Vassilevski.P2008} (see also \cite{Vassilevski.P2008}). In the nonlinear AMLI-cycle MG, $n$ steps of the CG method with the MG on coarser level as a preconditioner are applied to define the coarse-level solver. Under the assumption that the convergence factor of the V-cycle MG with bounded-level difference is bounded, the uniform convergence property of the nonlinear AMLI-cycle MG methods is shown, if $n$ is chosen to be sufficiently large.

As we can see, the parameter $n$ plays an important role in both the linear and nonlinear AMLI-cycle MG methods.  This parameter must be large enough to guarantee the uniform convergence even for problems with full regularity according to the theoretical results. However, we can expect uniform convergence in such cases for any $n \in \mathbb{Z}^{+}$, especially $n=1$, which partly  motivated the present work.
More specifically, we provide such a uniform convergence analysis of the nonlinear AMLI-cycle MG method. Under the standard assumptions for approximation and smoothing properties, we show that both the nonsymmetric (without post-smoothing) and the symmetric (with both pre and post-smoothing) nonlinear AMLI-cycle MG method converge uniformly for any $n \geq 1$, i.e.,
\begin{equation*}
\| v - \hat{B}^{ns}_k[A_kv] \|_{A_k}^2 \leq \delta \| v \|_{A_k}^2,  \qquad
\| v - \hat{B}_k[A_kv] \|_{A_k}^2  \leq \delta \| v \|_{A_k}^2,  
\end{equation*}
where $\hat{B}^{ns}_k$ and $\hat{B}_k$, defined by Algorithms~\ref{alg:ns-AMLI}~and~\ref{alg:AMLI} below, denote the nonsymmetric and symmetric nonlinear AMLI-cycle MG methods, respectively, and where the constant $0 < \delta <1$ is independent of level $k$.  We also prove the same uniform convergence under the assumption used in \cite{Notay.Y;Vassilevski.P2008}, i.e. the boundedness of the V-cycle MG  method with bounded-level difference. Via this proof, we generalized the results in~\cite{Notay.Y;Vassilevski.P2008}, which
show only that Krylov subspace iterative methods using $\hat{B}^{ns}_{k}[\cdot]$ and $\hat{B}_{k}[\cdot]$ as preconditioners converges uniformly.  This means that all the recursive calls of the Krylov subspace iterative method can only be performed on the coarse levels. On the finest level, we can just perform the smoothing steps and still have a uniformly convergent method. On the other hand, similar to MG methods, without the approximation and smoothing properties, we are not able to show uniform convergence for nonlinear AMLI-cycle MG.  However, we can compare the nonlinear AMLI-cycle MG method with V-cycle MG method, and show that nonlinear AMLI-cycle MG method is always better than the corresponding V-cycle MG method for any $n \geq 1$. For the nonsymmetric case, we can show that
\begin{equation*} 
\| v - \tilde{B}^{ns}_k[A_kv] \|_{A_k} \leq  \| v -  B^{ns}_kA_kv \|_{A_k}
\end{equation*}
where $B^{ns}_k$ denotes the nonsymmetric V-cycle MG (without post-smoothing), i.e., the~$\backslash$-cycle. For the symmetric case, under the assumption that the smoother is convergent in the~$\| \cdot \|_{A_k}$ norm, we have
\begin{equation*}
( v - \tilde{B}_k[A_kv], v)_{A_k} \leq  (v -  B_kA_kv, v)_{A_k},
\end{equation*}
where $B_k$ denotes the V-cycle MG. The above inequality is based on an important property of the full version of nonlinear preconditioned conjugate gradient (PCG) method; i.e., the residual of the current iteration is orthogonal to all the previous search directions. However, this property fails for the truncated version of the nonlinear PCG method. Therefore, the full version nonlinear PCG should be used to define the coarse level solver in the nonlinear AMLI-cycle MG method rather than the steepest descent method or any truncated version of the nonlinear PCG .  We also compare the nonlinear AMLI-cycle MG method with the corresponding $n$-fold V-cycle MG method and show that the nonlinear AMLI-cycle is always at least as good as and usually better than the $n$-fold V-cycle MG method in terms of the bounds on the convergence rate. 

The rest of the paper is organized as follows. In section~2, we introduce the nonlinear AMLI-cycle MG algorithms and the basic assumptions. The main results, the comparison theorem and the uniform convergence of the nonlinear AMLI-cycle MG method are presented in section~3. In section~4, the numerical experiments and the results that illustrate our theoretical results are presented.

\section{Preliminaries} \label{sec:pre}
Let~$V$~be a linear vector space, and let~$(\cdot, \cdot)$~denote a given inner product on~$V$, the induced norm of which is~$\| \cdot \|$. The adjoint of~$A$~with respect to~$(\cdot, \cdot)$, denoted by~$A^t$, is defined by~$(Au,v) = (u,A^tv)$~for all~$u,v\in V$. $A$~is SPD if~$A^t=A$~and~$(Av,v) > 0$~for all~$v \in V\backslash\{0\}$. As~$A$~is SPD with respect to~$(\cdot,\cdot)$, then~$(A\cdot,\cdot)$~defines another inner product on~$V$, denoted by~$(\cdot,\cdot)_{A}$, the induced norm of which is~$\| \cdot \|_{A}$.
\subsection{Multigrid}
Let us first introduce the standard V-cycle MG method. Here, we consider the MG methods that are based on a nested sequence of the subspaces of~$V$:
$
V_1 \subset V_2 \subset  \cdots \subset V_J = V. 
$ 
Corresponding to these spaces, $Q_k, P_k : V \rightarrow V_k$ are defined as the orthogonal projections with respect to~$(\cdot, \cdot)$~and~$(\cdot,\cdot)_A$, respectively, and define~$A_k: V_k \rightarrow V_k$ by $(A_ku_k,v_k) = (u_k,v_k)_A$ for~$u_k,v_k \in V_k$. Note that $A_k$ is also SPD; therefore, $A_k$~defines an inner product denoted by~$(\cdot,\cdot)_{A_k}$ on~$V_k$, the induced norm of which is~$\| \cdot \|_{A_k}$. We also introduce a smoother, $R_k: V_k \rightarrow V_k$, which is necessary to define the multigrid method.

Now we define the nonsymmetric multigrid iterator $B_k^{ns}$ (without post-smoothing) by the following recursive algorithm:

\begin{algorithm}[htbp]
\caption{$\backslash$-cycle MG:~$B_k^{ns}$} \label{alg:half-cycle}
\begin{flushleft}
 Let~$B_1^{ns} = A_1^{-1}$, and assume that~$B_{k-1}^{ns}:V_{k-1} \rightarrow V_{k-1}$~has been defined; therefore, for~$f \in V_k$, $B_k^{ns}: V_k \rightarrow V_k$~is defined as follow:
\end{flushleft}
\begin{enumerate}
\item [] {\bf Pre-smoothing:} $u_1 = R_k f$;
\item [] {\bf Coarse-grid correction:} $B_k^{ns}f := u_1 +  B_{k-1}^{ns}Q_{k-1}(f-A_k u_1)$.
\end{enumerate}
\end{algorithm}

Similarly, we can also define the (symmetric) V-cycle multigrid operator $B_k$ recursively, as  shown in Algorithm \ref{alg:V-cycle}. 
\begin{algorithm}[htpb!]
\caption{ V-cycle MG:~$B_k$} \label{alg:V-cycle}
\begin{flushleft}
 Let~$B_1 = A_1^{-1}$, and assume that~$B_{k-1}:V_{k-1} \rightarrow V_{k-1}$~has been defined; therefore, for~$f \in V_k$, $B_k: V_k \rightarrow V_k$~is defined as follow:
\end{flushleft}
\begin{enumerate}
\item [] {\bf Pre-smoothing} $u_1 = R_k f$;
\item [] {\bf Coarse-grid correction} $u_2= u_1 +  B_{k-1}Q_{k-1}(f-A_k u_1)$;
\item [] {\bf Post-smoothing} $B_kf:=u_2 + R_k^t(f-A_ku_2)$.
\end{enumerate}
\end{algorithm}

\subsection{Nonlinear preconditioned conjugate gradient method}
In order to introduce the nonlinear AMLI-cycle MG method, it is necessary to introduce the nonlinear PCG method, which is a simplified version (available for SPD $A_k$) of the algorithm originated in~\cite{Axelsson.O;Vassilevski.P1991a}. The original version 
in~\cite{Axelsson.O;Vassilevski.P1991a} was meant for more general cases, including nonsymmetric and possibly indefinite matrices. 
Let~$\hat{B}_k[\cdot]: V_k \rightarrow V_k$ be a given nonlinear operator intended  to approximate the inverse of~$A_k$. We now formulate the nonlinear PCG method used to provide an iterated approximate inverse to~$A_k$ based on the given nonlinear operator~$\hat{B}_k[\cdot]$. This procedure gives another nonlinear operator~$\tilde{B}^{[n]}_k[\cdot]: V_k \rightarrow V_k$, which can be viewed as an improved approximation of the inverse of~$A_k$.

\begin{algorithm}[htbp!]
\caption{Nonlinear PCG Method} \label{alg:nonlinear-CG}
\begin{flushleft}
Assume we are given a nonlinear operator~$\hat{B}_k[\cdot]$~to be used as a preconditioner.
Then, $\forall f \in V_k$, $\tilde{B}^{[n]}_k[f]$~is defined as follows:
\begin{enumerate}
\item [] Step 1. Let~$u_0 = 0$~and~$r_0 = f$. Compute~$p_0 = \hat{B}_k[r_0]$. Then let
\[
u_1 = \alpha_0 p_0, \ \text{and} \ r_1=r_0 - \alpha_0A_kp_0, \ \text{where} \ \alpha_0 = \frac{(r_0, p_0)}{(p_0,p_0)_{A_k}}.
\] 
\item [] Step 2. For~$i=1,2,\cdots,n-1$, compute the next conjugate direction
\begin{equation}\label{eqn:conjugate_direction}
p_i = \hat{B}_k[r_i] + \sum_{j=0}^{i-1} \beta_{i,j} p_j, \ \text{where} \ \beta_{i,j} = -\frac{(\hat{B}_k[r_i], p_j)_{A_k}}{(p_j,p_j)_{A_k}}.
\end{equation}
Then the next iterate is
\begin{equation}\label{eqn:ncg_iterate}
u_{i+1} = u_{i} + \alpha_i p_i, \ \text{where} \ \alpha_i = \frac{(r_i, p_i)}{(p_i,p_i)_{A_k}}, 
\end{equation}
and the corresponding residual is
\begin{equation}  \label{eqn:ncg_residual}
r_{i+1} = r_{i} - \alpha_i A_k p_i.
\end{equation}

\item [] Step 3. Let~$\tilde{B}^{[n]}_k[f] := u_n$.
\end{enumerate}
\end{flushleft}
\end{algorithm}

Algorithm~\ref{alg:nonlinear-CG} defines the nonlinear operator $\tilde{B}_k^{[n]}[\cdot]$.  In the rest of the paper, for the sake of simplicity, we will drop the superscript $[n]$ and use $\tilde{B}_k[\cdot]$ instead.  This simplified notation indicates that $n$ steps of the nonlinear PCG are performed. 

\begin{remark} \label{rem:1iter_CG}
If we apply only one step of the nonlinear PCG method, we can see that $ \tilde{B}_k[f] = \alpha \hat{B}_k[f]$~where $\alpha = (\hat{B}_k[f],f)/\| \hat{B}_k[f]  \|^2_{A_k}$. 
That is, $\tilde{B}_k[f]$~differs from~$\hat{B}_k[f]$ by a scalar factor.
\end{remark}

\begin{remark} \label{rem:CG_prop}
Due to the choice of $\beta_{i,j}$, it is easy to see that the new direction $p_i$ is $A_k$-orthogonal
 to all the previous directions $p_j$, $j=0, 1, \cdots, i-1$, i.e., 
\begin{equation}\label{eqn:p_A-orthogonal}
(p_i, p_j)_{A_k} = 0, \ j = 0,1,2, \cdots, i-1.
\end{equation}
Due to this property of the direction ${p_i}$ and to the choice of $\alpha_i$, from~\eqref{eqn:ncg_residual}, it is straightforward to see that 
\begin{equation} \label{eqn:rp_orthogonal}
(r_{i+1}, p_j) = 0, \ j = 0, 1,2, \cdots, i.
\end{equation}
Finally, by~\eqref{eqn:p_A-orthogonal}~and~\eqref{eqn:rp_orthogonal}, we can show that~$u_{i+1}$~computed by~\eqref{eqn:ncg_iterate}~is the solution of the minimization problem
$
\min_{\alpha_{i,j}\in \mathbb{R}} \| f - A_k(u_i + \sum_{j=0}^{i} \alpha_{i,j} p_j)  \|^2_{A_k^{-1}}.
$
Therefore, we have
$
\| f - A_k u_{i+1} \|^2_{A_k^{-1}} \leq \| f - A_k u_i \|^2_{A_k^{-1}}.
$
Then, by induction, we have
\begin{equation} \label{ine:improve}
\| A_k^{-1} f - \tilde{B}_k[f] \|^2_{A_k} \leq \| A_k^{-1}f - \hat{B}_k[f] \|^2_{A_k}.
\end{equation}
This means that $\tilde{B}_k[\cdot]$~is a better approximation to~$A_{k}^{-1}$~than to~$\hat{B}_k[\cdot]$. 
\end{remark}

\begin{remark} \label{rem:truncated}
According to equation~\eqref{eqn:conjugate_direction}, we use all the previous search directions in order to compute the next search direction. The resulting Algorithm~\ref{alg:nonlinear-CG}~is referred to as the full version of the nonlinear PCG method. In practice, due to memory constraints, 
we may want to use a truncated version. Specifically, we only require that the new direction to be 
orthogonal to the $m_i \ge 0$ most recent ones (cf. \cite{Notay.Y2000}). In that case, equation~\eqref{eqn:conjugate_direction}~is replaced by
$p_i = \hat{B}_k[r_i] + \sum_{j=i-1-m_i}^{i-1} \beta_{i,j} p_j$ where $\beta_{i,j} = -\frac{(\hat{B}_k[r_i], p_j)_{A_k}}{(p_j,p_j)_{A_k}}$,
and the resulting algorithm is referred to as the truncated version of the nonlinear PCG method.  A general strategy is  to use $0 \leq m_i \leq m_{i-1} + 1 \leq i-1$
and a typical choice is $m_i = 0$. If $p_i =\hat{B}_k[r_i]$ (i.e., formally $m_i = -1$), 
this choice corresponds to the nonlinear preconditioned steepest descent method. 
In the present multilevel setting, the full version of the  method is acceptable in practice, 
this is because we expect relatively few recursive calls (between the levels), and this happens on coarse levels.
\end{remark}

Assume that $\hat{B}_{k}[\cdot]$ approximates the inverse of $A_k$ with accuracy $\delta \in [0,1)$, i.e.,
\begin{equation} \label{ine:hatB_invA}
\| A_{k}^{-1} f - \hat{B}_{k}[f]\|_{A_k} \leq \delta \| f \|_{A_k^{-1}}.
\end{equation}
Then we have the following convergence result for the nonlinear PCG methods which we will use later.
\begin{theorem}[Theorem 10.2, \cite{Vassilevski.P2008}]\label{thm:converge-NPCG}
Assume that $\hat{B}_k[\cdot]$ satisfies~\eqref{ine:hatB_invA} and that $\tilde{B}_k[\cdot]$ is implemented by $n$ iterations of Algorithm~\ref{alg:nonlinear-CG} with $\hat{B}_k[\cdot]$ as the preconditioner.  Then the following convergence rate estimate holds:
\begin{equation}\label{ine:converge-NPCG}
\| A_{k}^{-1}f - \tilde{B}_k[f]  \|_{A_k} \leq \delta^n \| f \|_{A_k^{-1}}.
\end{equation}
\end{theorem}

\begin{remark}\label{rem:converge-NPCG}
As stated in Theorem 10.2 in \cite{Vassilevski.P2008}, the above convergence rate estimate holds for both the full and truncated versions of the nonlinear PCG method.
\end{remark}

\subsection{Nonlinear AMLI-cycle MG}
Now, thanks to Algorithm~\ref{alg:nonlinear-CG}, we can recursively construct the nonlinear AMLI-cycle MG operator as an approximation of~$A_k^{-1}$.  First, we define a nonsymmetric operator, i.e. a nonlinear AMLI-cycle MG without post-smoothing, as shown in Algorithm \ref{alg:ns-AMLI}.
\begin{algorithm}
\caption{Nonsymmetric nonlinear AMLI-cycle MG:~$\hat{B}^{ns}_{k}[\cdot]$} \label{alg:ns-AMLI}
\begin{flushleft}
Assume that $\hat{B}^{ns}_1[f]=A_1^{-1}f$ and that $\hat{B}^{ns}_{k-1}[\cdot]$ has been defined. Then for $f \in V_k$
\end{flushleft}
\begin{enumerate}
\item [] {\bf Pre-smoothing:} $u_1 = R_k f$; 
\item [] {\bf Coarse-grid correction:} $\hat{B}^{ns}_k[f] := u_1 +  \tilde{B}^{ns}_{k-1}[Q_{k-1}(f-A_k u_1)]$, where~$\tilde{B}^{ns}_{k-1}[\cdot]$~is implemented as in Algorithm~\ref{alg:nonlinear-CG} with~$\hat{B}^{ns}_{k-1}$~as the preconditioner. 
\end{enumerate}
\end{algorithm}
Similarly to standard (linear) MG, we can also define a symmetric nonlinear AMLI-cycle multigrid by introducing post-smoothing, as shown in Algorithm \ref{alg:AMLI}.
\begin{algorithm}
\caption{Nonlinear AMLI-cycle MG:~$\hat{B}_{k}[\cdot]$} \label{alg:AMLI}
\begin{flushleft}
Assume that $\hat{B}_1[f]=A_1^{-1}f$ and that $\hat{B}_{k-1}[\cdot]$ has been defined, then for $f \in V_k$,
\end{flushleft}
\begin{enumerate}
\item [] {\bf Pre-smoothing} $u_1 = R_k f$; 
\item [] {\bf Coarse-grid correction} $u_2 = u_1 +  \tilde{B}_{k-1}[Q_{k-1}(f-A_k u_1)]$, where~$\tilde{B}_{k-1}[\cdot]$~is implemented as in Algorithm~\ref{alg:nonlinear-CG} with~$\hat{B}_{k-1}$~as the preconditioner;
\item [] {\bf Post-smoothing} $\hat{B}_k[f] := u_2 + R_k^t(f - A_k u_2)$. 
\end{enumerate}
\end{algorithm}
On level $k$, once $\hat{B}^{ns}_k[\cdot]$ and $\hat{B}_k[\cdot]$ are defined, the corresponding $\tilde{B}^{ns}_k[\cdot]$ and $\tilde{B}_{k}[\cdot]$ are defined by $n$ steps of the nonlinear PCG method using $\hat{B}^{ns}_k[\cdot]$ and $\hat{B}_k[\cdot]$ as the preconditioner, respectively. 

\subsection{Cost of the Nonlinear AMLI-cycle MG}
The cost of the nonlinear AMLI-cycle is discussed in \cite{Axelsson.O;Vassilevski.P.1991a,Vassilevski.P2008,Notay.Y;Vassilevski.P2008}.  We include a complexity estimation here for the sake of completeness.  And we use the notation and terminology from \cite{Notay.Y;Vassilevski.P2008}.  

Let $n_k$ be the number of degrees of freedom at level $k$, and assume a uniform refinement, i.e., $n_k = \mu^d n_{k-1}$, $d=2$ or $d=3$ in which typically $\mu=2$.  Furthermore, assume that the V-cycle MG on level $k$ can be implemented for $\mathcal{O}(n_k)$ flops and that there are $n$ iterations of the nonlinear PCG on the coarse level. Then the cost of the nonlinear AMLI-cycle MG $\hat{B}_k[\cdot]$ can be estimated (using induction) by 
\begin{equation*}
w_k = \mathcal{O}(n_k) + n w_{k-1},
 \end{equation*}
which implies that
\begin{equation*}
w_k = \mathcal{O}(n_k) \sum_j ( \frac{n}{\mu^d})^j.
\end{equation*}
Therefore, the work of each nonlinear AMLI-cycle is comparable to that of the corresponding $n$-fold V-cycle MG.  For example, when $n=2$, the cost of the nonlinear AMLI-cycle MG method is roughly the same as that of W-cycle MG method.  Moreover, for the method to have an optimal complexity method, $n$ must satisfy that
$
n < \mu^d.
$ 
For example, when $d=3$ and $\mu=2$, we need $n < 8$ which is a mild restriction.  It should also be noted that, it is not necessary to apply the nonlinear PCG on each level,  In fact, several levels can be skipped, which leads to the condition
$
n < \mu^{d k_0},
$ 
in which $k_0$ denotes the number of skipped levels.  This is a very mild restriction if we choose $k_0$ sufficiently large. 

Next, we will consider~$\tilde{B}_k[\cdot]$. Iit is implemented by $n$ steps of the nonlinear PCG with $\hat{B}_k[\cdot]$ as the preconditioner, therefore, the cost is similar to $n$ steps of the nonlinear PCG with the $n$-fold V-cycle MG as the preconditioner.

\subsection{Assumptions}
Our goal is to analyze the  convergence of the  nonlinear AMLI-cycle MG using the same assumptions as in the 
conventional (classical) convergence analysis of MG. 

We make the following (classical) assumptions in order to carry out the convergence analysis. These assumptions will be used to show the uniform convergence of the 
nonlinear AMLI-cycle MG method. The first assumption is refereed as to the approximation property of the projection $P_k$. 
\begin{assumption}[Approximation Property]\label{asmp:approx}
\begin{equation}\label{ine:approx}
\|(I-P_{k-1})v\|_{A_{k}}^2 \leq \frac{c_1}{\rho(A_k)}\|A_{k}v\|^2, \quad \forall v \in V_k
\end{equation}
where $\rho(A_k)$ is the spectral radius of $A_k$, and $c_1$ is a constant independent of $k$. 
\end{assumption}
This assumption is commonly used in the MG literature, e.g., Assumption A.7 in \cite{Bramble.J1993},  the ``strong approximation property" assumption in \cite{Vassilevski.P2008}, and Assumption A7.1 in \cite{Xu.J1989}.
Our Assumption \ref{asmp:approx} holds (see, e.g., \cite{Xu.J1989, Vassilevski.P2008}) in the case of 
second-order elliptic problems with full regularity. 

Another common assumption refers to the smoothers. In this paper, we always assume that the (generally nonsymmetric)
smoother~$R_k$ is convergent in the~$\| \cdot \|_{A_k}$~norm.

One main assumption is that the symmetric composite smoother $\tilde{R}_k$, defined by 
 $
 I-\tilde{R}_kA_k = (I-R_kA_k)(I-R_k^tA_k),
 $
 satisfies the following smoothing property.
\begin{assumption}[Smoothing Property]\label{asmp:smooth}
\begin{equation}\label{ine:smooth}
\frac{c_2}{\rho (A_k)}(v,v) \leq (\tilde{R}_kv, v), \forall v \in V_k
\end{equation}
where $c_2$ is a constant independent of $k$.
\end{assumption}
This assumption means that the choice of smoother must be comparable to a simple Richardson smoother. 
It is used to prove estimates concerning the V-cycle MG method, see Assumption A.4. in \cite{Bramble.J1993}. 
Note that we also have another symmetric composite smoother~$\bar{R}_k$~which is defined by
 $
 I-\bar{R}_kA_k = (I-R_k^tA_k)(I-R_kA_k).
 $

Based on the Assumptions~\ref{asmp:approx}~and~\ref{asmp:smooth}, we obtain the following well-known result (see p. 75 of \cite{Xu.J1989} and p. 145 of \cite{Vassilevski.P2008}).
\begin{lemma} \label{lemma:approx}
Assume that Assumptions \ref{asmp:approx} and~\ref{asmp:smooth}~hold.  Then we have
$
\|(I-P_{k-1})\hat{v}\|^2_{A_k} \leq \eta ( \|v\|_{A_k}^2 - \|  \hat{v} \|^2_{A_k} )
$
where~$\hat{v} = (I-R_kA_k)v$, $v \in V$, and~$\eta = \frac{c_1}{c_2} > 0$ is a constant independent of $k$.
\end{lemma}

\begin{remark}
The above lemma can be found as Assumption (A) in \cite{Vassilevski.P2008} and Lemma 6.2 in \cite{Xu.J1989}.  It provides perhaps the shortest convergence proof for the V-cycle MG method. It is also equivalent to the assumption originally used in  \cite{McCormick.S1984,McCormick.S1985} (see \cite{Vassilevski.P2008} for details). Inequality~\eqref{ine:approx} can also be found as inequality (4.82) in \cite{Shauidurov.V1995}.
\end{remark}

\setcounter{remark-count}{1}
\section{Convergence Analysis}
In this section, we present the main results of this paper.  First, we compare the nonlinear AMLI-cycle MG method with the V-cycle MG method and thereby show that the nonlinear AMLI-cycle MG is always better or not worse than the respective V-cycle MG.  Furthermore, based on Assumptions~\ref{asmp:approx} and~\ref{asmp:smooth}, we show that the nonlinear AMLI-cycle MG method is uniformly convergent without the requirement that $n$, the number of iterations of the nonlinear PCG method, be sufficiently large.  

The following two representations are useful in our analysis. First, we have a result  for the nonsymmetric nonlinear operator~$\hat{B}_k^{ns}[\cdot]$ defined in Algorithm~\ref{alg:ns-AMLI}.
\begin{lemma} \label{lemma:B^ns}
For all $v \in V_k$,
\begin{equation}  \label{eqn:error_B^ns}
v - \hat{B}_k^{ns}[A_k v ] = (I - R_kA_k) v - \tilde{B}^{ns}_{k-1}[A_{k-1}P_{k-1}(I-R_kA_k)v]
\end{equation}
and
\begin{equation} \label{eqn:B^ns}
\hat{B}_k^{ns}[v] = R_kv + \tilde{B}_{k-1}^{ns}[Q_{k-1}(I-A_kR_k)v]. 
\end{equation}
\end{lemma}
\begin{proof}
Properties \eqref{eqn:error_B^ns}~and~\eqref{eqn:B^ns}~both follow directly from Algorithm~\ref{alg:ns-AMLI}~and the identity 
$A_{k-1}P_{k-1} = Q_{k-1}A_k$ that holds on $V_k$. 
\end{proof}

Similarly, we have the following lemma concerning the (symmetric) nonlinear operator $\hat{B}_k$ defined in Algorithm~\ref{alg:AMLI}.
\begin{lemma}\label{lemma:B}
For all $v \in V_k$,
\begin{equation} \label{eqn:error_B}
v - \hat{B}_k[A_k v] = (I-R_k^tA_k) ((I-R_kA_k)v - \tilde{B}_{k-1}[A_{k-1}P_{k-1}(I-R_kA_k)v])
\end{equation}
and 
\begin{equation} \label{eqn:B}
\hat{B}_k[v] = \bar{R}_kv + (I-R_k^tA_k)\tilde{B}_{k-1}[Q_{k-1}(I-A_kR_k)v].
\end{equation}
\end{lemma}
\begin{proof}
Properties \eqref{eqn:error_B}~and~\eqref{eqn:B} follow directly from the definition in Algorithm~\ref{alg:AMLI} (using again the 
identity $A_{k-1}P_{k-1} =  Q_{k-1}A_k$ that holds on $V_k$). 
\end{proof}

\subsection{Comparison Results Without Assumptions~\ref{asmp:approx}~and~\ref{asmp:smooth}} First, we compare the nonlinear AMLI-cycle MG method with the corresponding nonsymmetric ($\backslash$-cycle) MG and symmetric (V-cycle) MG method without Assumptions~\ref{asmp:approx}~and~\ref{asmp:smooth}.  We show that the nonlinear AMLI-cycle is always better (or not worse) under the assumption that the smoother is convergent in the $\| \cdot \|_{A_k}$-norm. 

The first comparison result is for the nonsymmetric nonlinear AMLI-cycle MG and the $\backslash$-cycle MG. 
The result shows that the nonlinear operator $\hat{B}^{ns}_k$ and $\tilde{B}^{ns}_k$ each give better approximations of the inverse of $A_k$ than does the $\backslash$-cycle MG.
\begin{theorem}\label{thm:ns_compare}
Let both $B^{ns}_k$ and $\hat{B}^{ns}_k[\cdot]$ be defined by Algorithms~\ref{alg:half-cycle}~and~\ref{alg:ns-AMLI}~, and let~$\tilde{B}_k^{ns}[\cdot]$ be implemented as in Algorithm~\ref{alg:nonlinear-CG} with~$\hat{B}_k^{ns}[\cdot]$~as the preconditioner. Then we have 
\begin{equation}
\| v - \tilde{B}^{ns}_k[A_kv] \|_{A_k} \leq \| v - \hat{B}^{ns}_k[A_kv] \|_{A_k} \leq  \| v -  B^{ns}_kA_kv \|_{A_k}\label{ine:ns-compare}
\end{equation}
\end{theorem}
\begin{proof}
We use mathematical induction to prove the theorem.  The result holds for $k=1$ trivially.  Assume that \eqref{ine:ns-compare} holds for $k-1$. By Algorithm~\ref{alg:half-cycle}, we have
$
(I-B^{ns}_kA_k)v = \hat{v} - B^{ns}_{k-1}A_{k-1}P_{k-1}\hat{v} = \hat{v} - P_{k-1}\hat{v}  + P_{k-1}\hat{v} - B^{ns}_{k-1}A_{k-1}P_{k-1}\hat{v}
$
where~$\hat{v} = (I- R_kA_k)v$.  Note that~$P_{k-1}$~is a projection; hence, we have
$
\| v - B^{ns}_kA_kv \|_{A_k}^2  =  \| \hat{v} - P_{k-1}\hat{v} \|_{A_k}^2 + \| P_{k-1}\hat{v} - B^{ns}_{k-1}A_{k-1}P_{k-1}\hat{v}\|_{A_k}^2.
$
Similarly, for the nonlinear operator~$\hat{B}^{ns}_k[\cdot]$, by Lemma~\ref{lemma:B^ns}, we have
\begin{align*}
\| v - \hat{B}^{ns}_k[A_kv] \|_{A_k}^2 & =  \| \hat{v} - P_{k-1}\hat{v} \|_{A_k}^2 + \| P_{k-1}\hat{v} - \tilde{B}^{ns}_{k-1}[A_{k-1}P_{k-1}\hat{v}] \|_{A_k}^2  \\
& \leq  \| \hat{v} - P_{k-1}\hat{v} \|_{A_k}^2 +  \| P_{k-1}\hat{v} - B^{ns}_{k-1}A_{k-1}P_{k-1}\hat{v}\|_{A_k}^2  \\
& = \| v - B^{ns}_kA_kv \|_{A_k}^2.
\end{align*} 
Moreover, therefore, $\tilde{B}_k^{ns}[A_kv]$~is obtained by Algorithm~\ref{alg:nonlinear-CG}~with $\hat{B}_k^{ns}[\cdot]$ as the preconditioner, by~\eqref{ine:improve}, we have
$
\|v - \tilde{B}^{ns}_k[A_k v ] \|_{A_k}^2 \leq \|v - \hat{B}^{ns}_k[A_k v ] \|_{A_k}^2 \leq \| v - B^{ns}_kA_kv \|_{A_k}^2.
$
This completes the proof.
\end{proof}

Note that we only used the minimization property~\eqref{ine:improve} in the proof.  Therefore, Theorem~\ref{thm:ns_compare}~also holds when we use any truncated version of the nonlinear PCG method is used to define the coarse level solver. Thus, we have the following corollary.
\begin{corollary}\label{coro:ns_compare}
Let both~$B^{ns}_k$ and $\hat{B}^{ns}_k[\cdot]$ be defined by Algorithms~\ref{alg:half-cycle}~and~\ref{alg:ns-AMLI}.  Also, let~$\tilde{B}_k^{ns}[\cdot]$~be implemented as in a truncated version of Algorithm~\ref{alg:nonlinear-CG} with $\hat{B}_k^{ns}[\cdot]$~as the preconditioner. Then, we have 
\begin{align*} 
\| v - \tilde{B}^{ns}_k[A_kv] \|_{A_k}  \leq \| v - \hat{B}^{ns}_k[A_kv] \|_{A_k} \leq  \| v -  B^{ns}_kA_kv \|_{A_k}.  
\end{align*}
\end{corollary}

Next we show that, similar to the nonsymmetric case, the nonlinear AMLI-cycle MG method is better (not worse) than the respective V-cycle MG method, and the nonlinear operators $\hat{B}_k$ and $\tilde{B}_k$ each provides a better approximations of the inverse of $A_k$. 

We need the following key property of the nonlinear operator~$\tilde{B}_k[\cdot]$~obtained by Algorithm~\ref{alg:nonlinear-CG}. This property plays an important role in our analysis.
\begin{lemma}\label{lemma:tildeB}
Let~$\tilde{B}_k[\cdot]$~be implemented as in Algorithm~\ref{alg:nonlinear-CG}~with~$\hat{B}_k[\cdot]$~as the preconditioner. For~$\forall v \in V_k$, we have
\begin{equation}\label{eqn:tildeB}
\| v -  \tilde{B}_k[A_kv] \|^2_{A_k} = (v - \tilde{B}_k[A_kv], v)_{A_k}.
\end{equation}
\end{lemma}
\begin{proof}
By~\eqref{eqn:ncg_iterate}, we can see that~$u_i = \sum_{j=0}^{i-1}\alpha_{j}p_j$. Due to the fact that the residual~$r_i$~is orthogonal to all the previous directions~$p_j$, $j=0, 1, \cdots, i-1$, we have~$(r_i,u_i)=0$. By definition, $\tilde{B}_k[f] := u_{n}$.  Therefore, we have
$(r_n,u_n) =0$, $r_n = f - A_k \tilde{B}_k[f]$, i.e., 
\begin{equation} \label{eqn:ru_orthogonal}
(f-A_k\tilde{B}_k[f], \tilde{B}_k[f])=0.
\end{equation}
If we let $f = A_kv$, we have
\begin{align*}
\| v - \tilde{B}_k[A_kv] \|_{A_k}^2 &= (v - \tilde{B}_k[A_kv], v-\tilde{B}_k[A_kv])_{A_k} \\
	& = (v-\tilde{B}_k[A_kv], v)_{A_k} - (v-\tilde{B}_k[A_kv],\tilde{B}_k[A_kv])_{A_k}. 
\end{align*}
The second term vanishes due to~\eqref{eqn:ru_orthogonal}~and the choice~$f = A_kv$. Then~\eqref{eqn:tildeB}~follows directly.
\end{proof}

Now we are in a position to show the following comparison theorem for the nonlinear AMLI-cycle MG method and the respective V-cycle MG method.
\begin{theorem}\label{thm:AMLI_vs_V}
Let ~$\hat{B}_k[\cdot]$~be defined by Algorithm~\ref{alg:AMLI}, and let~$\tilde{B}_k[\cdot]$~be implemented as in Algorithm~\ref{alg:nonlinear-CG} with~$\hat{B}_k[\cdot]$~as the preconditioner. We also assume that the smoother~$R_k$~is convergent. For~$\forall v \in V_k$, the following estimates hold:
\begin{equation} \label{ine:AMLI_vs_V}
0 \leq (v - \tilde{B}_k[A_kv], v)_{A_k} \leq (v - \hat{B}_k[A_kv], v)_{A_k} \leq (v - B_kA_kv, v)_{A_k}.
\end{equation}
\end{theorem}
\begin{proof}
Inequalities \eqref{ine:AMLI_vs_V}~hold trivially for~$k=1$. Assuming by induction that~\eqref{ine:AMLI_vs_V} holds for~$k-1$, by Lemma~\ref{lemma:tildeB}, we then have that
\begin{equation*}
(v - \tilde{B}_k[A_kv], v)_{A_k}  = \| v -  \tilde{B}_k[A_kv] \|^2_{A_k}  \geq 0,
\end{equation*}
which confirms the first inequality in~\eqref{ine:AMLI_vs_V}.  As~$\tilde{B}_k[A_kv]$~is obtained by Algorithm~\ref{alg:nonlinear-CG} with $\hat{B}_k[\cdot]$ as the preconditioner, by~\eqref{ine:improve}, we have that
\begin{equation*}
(v - \tilde{B}_k[A_kv], v)_{A_k}  = \|v - \tilde{B}_k[A_k v ] \|_{A_k}^2 \leq \|v - \hat{B}_k[A_k v ] \|_{A_k}^2. 
\end{equation*}
On the other hand, if we let~$\hat{v} = (I-R_kA_k)v$, then according to Lemma~\ref{lemma:B}, we have
\begin{align*}
\|v - \hat{B}_k[A_k v ] \|_{A_k}^2 & = \| (I-R_k^tA_k)(\hat{v} - \tilde{B}_{k-1}[A_{k-1}P_{k-1}\hat{v}])\|^2_{A_k} \\
(\text{smoother is convergent}) & \leq  \| \hat{v} - \tilde{B}_{k-1}[A_{k-1}P_{k-1}\hat{v}]\|^2_{A_k} \\
(\text{orthogonality})& = \| \hat{v} - P_{k-1}\hat{v} \|^2_{A_k} + \| P_{k-1} \hat{v} -  \tilde{B}_{k-1}[A_{k-1}P_{k-1}\hat{v}]\|^2_{A_k}  \\
(\text{Lemma~\ref{lemma:tildeB}})& = (\hat{v}-P_{k-1}\hat{v}, \hat{v}-P_{k-1}\hat{v})_{A_k} \\
& \quad + (P_{k-1} \hat{v} -  \tilde{B}_{k-1}[A_{k-1}P_{k-1}\hat{v}], P_{k-1}\hat{v})_{A_k} \\
(\text{orthogonality}) & = (\hat{v}-P_{k-1}\hat{v}, \hat{v})_{A_k} + (P_{k-1} \hat{v} -  \tilde{B}_{k-1}[A_{k-1}P_{k-1}\hat{v}], \hat{v})_{A_k} \\
(\text{Lemma~\ref{lemma:B}})& = (v - \hat{B}_k[A_kv],v)_{A_k}.
\end{align*}
Therefore, we have
$
(v - \tilde{B}_k[A_kv], v)_{A_k}\leq (v - \hat{B}_k[A_kv],v)_{A_k},
$
which confirms the second inequality in~\eqref{ine:AMLI_vs_V}. For the last inequality, we have
\begin{align*}
(v - \hat{B}_k[A_kv],v)_{A_k} & = (\hat{v}-P_{k-1}\hat{v}, \hat{v})_{A_k} + (P_{k-1} \hat{v} -  \tilde{B}_{k-1}[A_{k-1}P_{k-1}\hat{v}], \hat{v})_{A_k} \\
(\text{induction assumption}) & \leq (\hat{v}-P_{k-1}\hat{v}, \hat{v})_{A_k} + (P_{k-1} \hat{v} -  B_{k-1}A_{k-1}P_{k-1}\hat{v}, P_{k-1}\hat{v})_{A_k} \\ 
& = (v - B_kA_kv,v)_{A_k}.
\end{align*}
This confirms the last inequality in~\eqref{ine:AMLI_vs_V}.  This completes the proof. 
\end{proof}

\begin{remark}
We recall that Lemma~\ref{lemma:tildeB}~is based on the fact that the current residual $r_i$ is orthogonal to all previous search directions--a fact that only holds for the full version of the nonlinear AMLI-cycle MG. 
Therefore, the full version of the nonlinear PCG should be preferred in practice over both the steepest descent method and the truncated version of the nonlinear PCG.  By choosing the full version of the nonlinear PCG, we will guarantee the monotonicity stated in  Theorem~\ref{thm:AMLI_vs_V}, which also holds only for this method.
\end{remark}

Regarding the comparison between the nonlinear AMLI-cycle MG method and the $n$-fold V-cycle MG method, the latter of which is defined by recursively applying the coarse-grid correction $n$ times (e.g., $n=2$ corresponds to the well-known W-cycle MG), results that are similar to~\eqref{ine:ns-compare} and~\eqref{ine:AMLI_vs_V}~are too strong and in general do not hold.  Because in general if 
$\|(I - B^1_k A_k)v \|_{A_k} \leq \|(I-B^2_kA_k)v \|_{A_k}$ holds for any $v \in V_k$ where both $B^1_k$ and $B^2_k$ are linear operators, this does not imply that $\| (I-B^1_k A_k )^n v  \|_{A_k} \leq \| (I - B^2_kA_k)^n v \|_{A_k}$ for all $v \in V_k$.  However, we can still show that the nonlinear AMLI-cycle MG method performs better (or not worse) 
(in terms of the convergence bounds)
than the corresponding $n$-fold V-cycle MG method for $n \geq 2$ in the following sense. 

\begin{theorem}\label{thm:ns_compare_n}
Let $B^{ns,[n]}_k$ be the $n$-fold V-cycle MG without post-smoothing and $\hat{B}^{ns}_k[\cdot]$ be defined by Algorithm~\ref{alg:ns-AMLI}, and let $\tilde{B}_k^{ns}[\cdot]$ be implemented as in Algorithm~\ref{alg:nonlinear-CG} with $n$ steps with~$\hat{B}_k^{ns}[\cdot]$~as the preconditioner.  Then we have
\begin{align*} 
\| v - B^{ns,[n]}_kA_kv \|^2_{A_k} &\leq  \rho_k \| v  \|^2_{A_k}, \\
\| v - \hat{B}^{ns}_k[A_kv] \|^2_{A_k} &\leq \delta_k  \| v \|^2_{A_k},   
\end{align*}
where $0 \leq \delta_k \leq \rho_k <1$. Equivalently, $\| v - \hat{B}^{ns}_k[A_kv] \|_{A_k} \leq \| I- B^{ns,[n]}_kA_k \|_{A_k} \| v \|_{A_k}$.
\end{theorem}
\begin{proof}
Because a direct solver is used on the coarsest level $k=1$, we have $0 = \delta_1 = \rho_1 < 1$.  Therefore, it is easy to see that $B_2^{ns, [n]} = \hat{B}_2^{ns}[\cdot]$, which means that $0 \leq \delta_2 = \rho_2 <1$.  This implies that $\tilde{B}_2^{ns}[\cdot]$ is defined as $n$ steps of the preconditioned Krylov method with $B_2^{ns,[n]}$ as the preconditioner.  By the minimization property of Algorithm~\ref{alg:nonlinear-CG} and the fact that the preconditioner is now linear, we have $ \| v - \tilde{B}_2^{ns}[A_2v] \|_{A_2} \leq \| (I-B_2^{ns,[n]}A_2)^n v|_{A_2}$.  Therefore, on level $3$, we have
\begin{align*}
\|v - \hat{B}^{ns}_3[A_3 v ] \|_{A_3}^2 &=  \| \hat{v} - P_{2}\hat{v} \|_{A_3}^2  + \| P_{2}\hat{v} - \tilde{B}^{ns}_{2}[A_{2}P_{2}\hat{v}] \|_{A_3}^2 \\
	& \leq \| \hat{v} - P_{2}\hat{v} \|_{A_3}^2  + \| (I - B^{ns,[n]}_{2}A_{2})^nP_{2}\hat{v} \|_{A_3}^2 \\
	& =  \| (I - B_3^{ns,[n]}A_3) v\|_{A_3}^2,
\end{align*}
where $\hat{v} = (I-R_kA_k)v$, $k=3$. On the other hand,
\begin{align*}
 \| (I - B_3^{ns,[n]}A_3) v\|_{A_3}^2 &= \| \hat{v} - P_{2}\hat{v} \|_{A_3}^2  + \| (I - B^{ns,[n]}_{2}A_{2})^nP_{2}\hat{v} \|_{A_3}^2 \\ 
 & \leq \| \hat{v} - P_{2}\hat{v} \|_{A_3}^2 + \rho^n_{2} \| P_{2} \hat{v} \|^2_{A_3}  \\
 			& \leq (1 - \rho_{2}^n) \| \hat{v} - P_{2}\hat{v} \|_{A_3}^2 + \rho_{2}^n \| \hat{v} \|^2_{A_3} \\
						& \leq ( (1-\rho_{2}^n) \delta_3^{TG} + \rho_{2}^n )  \| v \|^2_{A_k} \\
						& := \rho_3   \| v \|^2_{A_k},
\end{align*}
where $\rho_3 = (1-\rho_{2}^n) \delta_3^{TG} + \rho_{2}^n $ and where $\delta_3^{TG}$ is the two-grid convergence rate at level $3$. In the following, $\delta_k^{TG}$ denotes the two-grid convergence rate at level $k$.  Hence,  we have $ \|v - \hat{B}^{ns}_3[A_3 v ] \|_{A_3}^2 \leq \delta_3 \| v \|_{A_3}^2 $ and $0\leq \delta_3 \leq \rho_3 <1$.

For $k \geq 4$, we use mathematical induction.  Assume that $0 \leq \delta_{k-1} \leq \rho_{k-1} \leq 1$ holds; therefore, we have
$
(I-B^{ns,[n]}_kA_k)v  = \hat{v} - P_{k-1}\hat{v}  + (I- B^{ns,[n]}_{k-1}A_{k-1})^nP_{k-1}\hat{v}.
$
Note that~$P_{k-1}$~is a projection.  Therefore, we have
\begin{align*}
\| v - B^{ns,[n]}_kA_kv \|_{A_k}^2  &=  \| \hat{v} - P_{k-1}\hat{v} \|_{A_k}^2 + \| (I - B^{ns,[n]}_{k-1}A_{k-1})^nP_{k-1}\hat{v}\|_{A_k}^2 \\
\text{(Induction Assumption)}	& \leq \| \hat{v} - P_{k-1}\hat{v} \|_{A_k}^2 + \rho^n_{k-1} \| P_{k-1} \hat{v} \|^2_{A_k}  \\
\text{(Orthogonality)}			& \leq (1 - \rho_{k-1}^n) \| \hat{v} - P_{k-1}\hat{v} \|_{A_k}^2 + \rho_{k-1}^n \| \hat{v} \|^2_{A_k} \\
						& \leq ( (1-\rho_{k-1}^n) \delta_{k}^{TG} + \rho_{k-1}^n )  \| v \|^2_{A_k} \\
						& := \rho_k   \| v \|^2_{A_k},
\end{align*}
where $\rho_k = (1-\rho_{k-1}^n) \delta_k^{TG} + \rho_{k-1}^n $.  Similarly, for $\hat{B}^{ns}_k[\cdot]$, we have
\begin{align*}
\|v - \hat{B}^{ns}_k[A_k v ] \|_{A_k}^2 &=  \| \hat{v} - P_{k-1}\hat{v} \|_{A_k}^2  + \| P_{k-1}\hat{v} - \tilde{B}^{ns}_{k-1}[A_{k-1}P_{k-1}\hat{v}] \|_{A_k}^2 \\
\text{(Induction Assumption)}	&  \leq  \| \hat{v} - P_{k-1}\hat{v} \|_{A_k}^2  + \delta^n_{k-1} \| P_{k-1}\hat{v} \|_{A_k}^2  \\
						& \leq ((1-\delta^n_{k-1})  \delta_k^{TG} + \delta^n_{k-1}) \| v \|_{A_k}^2   \\
						& := \delta_k \| v \|_{A_k}^2,
\end{align*}
where $\delta_k = (1-\delta_{k-1}^n)  \delta_{k}^{TG} + \delta_{k-1}^n $.  Because $0 \leq  \delta_{k-1} \leq \rho_{k-1}<1$, it follows
that 
$0 \leq  \delta_{k} \leq \rho_{k}<1$.  This completes the proof.
\end{proof}

For the $n$-fold V-cycle MG method with post-smoothing, we have the following results.
\begin{theorem}\label{thm:compare_n}
Let $B^{[n]}_k$ be the $n$-fold V-cycle MG method with post-smoothing and $\hat{B}_k[\cdot]$ defined by Algorithm~\ref{alg:AMLI}, and let $\tilde{B}_k[\cdot]$ be implemented as in Algorithm~\ref{alg:nonlinear-CG} with $n$ steps with~$\hat{B}_k[\cdot]$~as the preconditioner.  Then we have
\begin{align*} 
\| v - B^{[n]}_kA_kv \|^2_{A_k} &\leq  \rho_k \| v  \|^2_{A_k}, \\
\| v - \hat{B}_k[A_kv] \|^2_{A_k} &\leq \delta_k  \| v \|^2_{A_k},   
\end{align*}
where $0 \leq \delta_k \leq \rho_k <1$. Equivalently, $\| v - \hat{B}_k[A_kv] \|_{A_k} \leq \| I- B^{[n]}_kA_k \|_{A_k} \| v \|_{A_k}$.
\end{theorem}
\begin{proof}
Due to the fact that we are using a direct solver on the coarsest grid, we have $B_2^{[2]} = \hat{B}_2[\cdot]$ again.  By the same argument as the proof of the previous theorem, we have $0 = \delta_1 = \rho_1$, $0 \leq \delta_2 = \rho_2 < 1$, and $0 \leq \delta_3 \leq \rho_3 <1$. 
For $k \geq 4$, assume that $0 \leq  \delta_{k-1} \leq \rho_{k-1}<1$.  Denote~$(I-R_kA_k)v$~by~$\hat{v}$ as before and~$\hat{w} = (I-R_kA_k)w$.  Therefore, we have
\begin{align*}
( v - B^{[n]}_kA_kv, w)_{A_k} &= (\hat{v} - P_{k-1} \hat{v} + (I - B^{[n]}_{k-1}A_{k-1})^nP_{k-1}\hat{v}, \hat{w})_{A_k} \\
\text{(Cauchy Schwarz)}& \leq \| \hat{v} - P_{k-1}\hat{v} \|_{A_k} \| \hat{w} - P_{k-1}\hat{w}\|_{A_k} \\ 
& \quad  + \| (I - B^{[n]}_{k-1}A_{k-1})^nP_{k-1}\hat{v} \|_{A_k} \| P_{k-1}\hat{w}\|_{A_k} \\
\text{(Induction Assumption)}& \leq \| \hat{v} - P_{k-1}\hat{v} \|_{A_k} \| \hat{w} - P_{k-1}\hat{w}\|_{A_k} \\
& \quad + \rho_{k-1}^{n/2} \| P_{k-1}\hat{v} \|_{A_k} \| P_{k-1}\hat{w}\|_{A_k} \\
\text{(Cauchy Schwarz)} & \leq \sqrt{\| \hat{v} - P_{k-1} \hat{v} \|_{A_k}^2 + \rho_{k-1}^n \| P_{k-1}\hat{v}  \|_{A_k}^2}  \\
& \quad  \times \sqrt{\| \hat{w}-P_{k-1} \hat{w} \|_{A_k}^2 + \| P_{k-1} \hat{w} \|_{A_k}^2 }.
\end{align*}
Then, similar to Theorem~\ref{thm:ns_compare_n}, we have
$
(v - B_k^{[n]}A_kv,w)_{A_k} \leq \rho_k^{1/2}\| v \|_{A_k} \| w \|_{A_k},
$
which implies that $ \| v - B^{[n]}_k A_k v \|^2_{A_k} \leq \rho_k \| v \|^2_{A_k} $ where~$\rho_k = (1-\rho_{k-1}^n) \delta_k^{TG} + \rho_{k-1}^n $.  On the other hand, we have
\begin{align*}
( v - \hat{B}_k[A_kv], w)_{A_k} 
& = (\hat{v} - P_{k-1}\hat{v}, \hat{w}-P_{k-1}\hat{w})_{A_k} \\
& \quad + (P_{k-1}\hat{v} - \tilde{B}_{k-1}[A_{k-1}P_{k-1}\hat{v}], P_{k-1}\hat{w})_{A_k} \\
\text{(Cauchy Schwarz)}& \leq \| \hat{v} - P_{k-1}\hat{v} \|_{A_k} \| \hat{w} - P_{k-1}\hat{w}\|_{A_k} \\ 
& \quad  + \| P_{k-1} \hat{v} - \tilde{B}_{k-1}[A_{k-1}P_{k-1}\hat{v}] \|_{A_k} \| P_{k-1}\hat{w}\|_{A_k} \\
\text{(Induction Assumption)}& \leq \| \hat{v} - P_{k-1}\hat{v} \|_{A_k} \| \hat{w} - P_{k-1}\hat{w}\|_{A_k} \\
& \quad + \delta_{k-1}^{n/2} \| P_{k-1}\hat{v} \|_{A_k} \| P_{k-1}\hat{w}\|_{A_k} \\
\text{(Cauchy Schwarz)} & \leq \sqrt{\| \hat{v} - P_{k-1} \hat{v} \|_{A_k}^2 + \delta_{k-1}^n \| P_{k-1}\hat{v}  \|_{A_k}^2}  \\
& \quad  \times \sqrt{\| \hat{w}-P_{k-1} \hat{w} \|_{A_k}^2 + \| P_{k-1} \hat{w} \|_{A_k}^2 }.
\end{align*} 
Then, we have
$
(v - \hat{B}_kA_kv,w)_{A_k} \leq \delta_k^{1/2}\| v \|_{A_k} \| w \|_{A_k},
$
which implies that 
$$ \| v - \hat{B}_k A_k v \|^2_{A_k} \leq \delta_k \| v \|^2_{A_k} $$ 
where~$\delta_k = (1-\delta_{k-1}^n) \delta^{TG}_k + \delta_{k-1}^n $.  Because $0 \leq  \delta_{k-1} \leq \rho_{k-1}<1$, it follows that $0 \leq  \delta_{k} \leq \rho_{k}<1$.  This completes the proof.
\end{proof}

Again, it is easy to see that the above results also hold for the truncated version of the nonlinear PCG method.  In addition, we have the following corollaries. 
\begin{corollary}
Let $B^{ns,[n]}_k$ be the $n$-fold V-cycle MG without post-smoothing. Define $\hat{B}^{ns}_k[\cdot]$ by Algorithm~\ref{alg:ns-AMLI}, and let $\tilde{B}_k^{ns}[\cdot]$ be implemented as in a truncated version of  Algorithm~\ref{alg:nonlinear-CG} with $n$ steps with~$\hat{B}_k^{ns}[\cdot]$~as the preconditioner.  Then we have
\begin{equation*} 
\| v - B^{ns,[n]}_kA_kv \|^2_{A_k} \leq  \rho_k \| v  \|^2_{A_k}, \qquad
\| v - \hat{B}^{ns}_k[A_kv] \|^2_{A_k} \leq \delta_k  \| v \|^2_{A_k},   
\end{equation*}
where $0 \leq \delta_k \leq \rho_k <1$. Equivalently, $\| v - \hat{B}^{ns}_k[A_kv] \|_{A_k} \leq \| I- B^{ns,[n]}_kA_k \|_{A_k} \| v \|_{A_k}$.
\end{corollary}

\begin{corollary}
Let $B^{[n]}_k$ be the $n$-fold V-cycle MG with post-smoothing. Define $\hat{B}_k[\cdot]$ by Algorithm~\ref{alg:AMLI}, and let $\tilde{B}_k[\cdot]$ be implemented as in a truncated version of Algorithm~\ref{alg:nonlinear-CG} with $n$ steps with~$\hat{B}_k[\cdot]$~as the preconditioner.  Then we have
\begin{equation*} 
\| v - B^{[n]}_kA_kv \|^2_{A_k} \leq  \rho_k \| v  \|^2_{A_k}, \qquad
\| v - \hat{B}_k[A_kv] \|^2_{A_k} \leq \delta_k  \| v \|^2_{A_k},   
\end{equation*}
where $0 \leq \delta_k \leq \rho_k <1$.  Equivalently, $\| v - \hat{B}_k[A_kv] \|_{A_k} \leq \| I- B^{[n]}_kA_k \|_{A_k} \| v \|_{A_k}$.
\end{corollary}

\subsection{Comparison Results under Assumptions~\ref{asmp:approx}~and~\ref{asmp:smooth}}
We return now to Assumptions~\ref{asmp:approx}~and~\ref{asmp:smooth}.
Under these assumptions, we have the following comparison theorem, which shows that the nonlinear AMLI-cycle MG method is better than the~$\backslash$-cycle MG method
by a factor of $\rho <1$, as specified in the following theorem.
\begin{theorem}\label{thm:AMLI_vs_ns}
Let~$\hat{B}_k[\cdot]$~be defined by Algorithm~\ref{alg:AMLI}, and let~$\tilde{B}_k[\cdot]$~be implemented as in Algorithm~\ref{alg:nonlinear-CG} with~$\hat{B}_k[\cdot]$~as the preconditioner. 
Assume that Assumptions~\ref{asmp:approx}~and~\ref{asmp:smooth} hold. We then have the estimates
\begin{equation}
\| v - \tilde{B}_k[A_kv] \|_{A_k} \leq \| v - \hat{B}_k[A_kv] \|_{A_k} \leq \rho \| v - B_{k}^{ns}A_k v \|_{A_k}  \label{ine:vs_ns} \\
\end{equation}
where~$\rho = \sqrt{\frac{c_1}{c_1+c_2}} < 1$, which is a constant independent of~$k$.
\end{theorem}
\begin{proof}
The results holds for $k=1$ trivially.  Assume that~\eqref{ine:vs_ns}~holds for level~$k-1$. Denote~$(I-R_kA_k)v$~by~$\hat{v}$, and let~$\hat{w} = (I-R_kA_k)w$. Similar to Theorem~\ref{thm:compare_n}, we have 
\begin{align*}
( v - \hat{B}_k[A_kv], w)_{A_k} 
& \leq \| \hat{v} - P_{k-1}\hat{v} \|_{A_k} \| \hat{w} - P_{k-1}\hat{w}\|_{A_k} \\
  	& \quad + \| P_{k-1} \hat{v} - \tilde{B}_{k-1}[A_{k-1}P_{k-1}\hat{v}] \|_{A_k} \| P_{k-1}\hat{w}\|_{A_k} \\
\text{(induction assumption)}& \leq \| \hat{v} - P_{k-1}\hat{v} \|_{A_k} \| \hat{w} - P_{k-1}\hat{w}\|_{A_k} \\
	& \quad+ \rho \| P_{k-1} \hat{v} - \tilde{B}^{ns}_{k-1}A_{k-1}P_{k-1}\hat{v} \|_{A_k} \| P_{k-1}\hat{w}\|_{A_k} \\
\text{(Cauchy-Schwarz)} & \leq \sqrt{\| \hat{v} - P_{k-1} \hat{v} \|_{A_k}^2 + \| P_{k-1}\hat{v} - B_{k-1}^{ns}A_{k-1}P_{k-1}\hat{v}\|_{A_k}^2} \\
	      & \quad \times \sqrt{\| \hat{w}-P_{k-1} \hat{w} \|_{A_k}^2 + \rho^2 \| P_{k-1} \hat{w} \|_{A_k}^2 } \\
	      & = \| v - B_k^{ns}A_k v \|_{A_k} \times \sqrt{\| \hat{w}-P_{k-1} \hat{w} \|_{A_k}^2 + \rho^2 \| P_{k-1} \hat{w} \|_{A_k}^2 }.
\end{align*}
Note that 
\begin{align*}
\| \hat{w}-P_{k-1} \hat{w} \|_{A_k}^2 + \rho^2 \| P_{k-1} \hat{w} \|_{A_k}^2 & = (1-\rho^2) \| \hat{w}-P_{k-1} \hat{w} \|_{A_k}^2 + \rho^2 \| \hat{w} \|_{A_k}^2 \\
\text{(Lemma~\ref{lemma:approx})}	& \leq (1-\rho^2) \eta(\| w \|^2_{A_k} - \| \hat{w} \|_{A_k}^2) + \rho^2 \| \hat{w} \|^2_{A_k} \\
& = \rho^2 \| w \|_{A_k}^2  \quad (\text{choose} \ \rho^2 = \frac{\eta}{1+\eta}).
\end{align*}
Therefore, we have
$
( v - \hat{B}_k[A_kv], w)_{A_k} \leq  \| v - B_k^{ns}A_k v \|_{A_k} \times \rho \| w \|_{A_k}.
$
Moreover, $\tilde{B}_k[A_kv]$~is obtained by Algorithm~\ref{alg:nonlinear-CG}~with $\hat{B}_k[\cdot]$ as the preconditioner, such that by~\eqref{ine:improve}, we have
$
\|v - \tilde{B}_k[A_k v ] \|_{A_k}^2 \leq \|v - \hat{B}_k[A_k v ] \|_{A_k}^2.
$
Then~\eqref{ine:vs_ns}~follows.
\end{proof}

\subsection{Uniform Convergence under  Assumptions~\ref{asmp:approx}~and~\ref{asmp:smooth}}
Now, we show the uniform convergence of the nonlinear AMLI-cycle MG method under Assumptions~\ref{asmp:approx}~and~\ref{asmp:smooth}.  Under these assumptions, it is well known that the $\backslash$-cycle MG method and the V-cycle MG methods are both uniformly convergent (see, e.g. \cite{Xu.J1989,Vassilevski.P2008}).  Therefore, thanks to the comparison results in the previous sections, the uniform convergence of the nonlinear AMLI-cycle MG method follows directly. 

Next the theorem shows the uniform convergence of the nonsymmetric nonlinear AMLI-cycle MG method under Assumptions~\ref{asmp:approx}~and~\ref{asmp:smooth}.
\begin{theorem} \label{thm:ns_uniform}
Let~$\hat{B}_k^{ns}[\cdot]$~be defined by Algorithm~\ref{alg:ns-AMLI}, and let~$\tilde{B}_k^{ns}[\cdot]$~be implemented as in Algorithm~\ref{alg:nonlinear-CG} with~$\hat{B}_k^{ns}[\cdot]$~as the preconditioner. Assume that  Assumptions~\ref{asmp:approx}~and~\ref{asmp:smooth} hold.  Then we have the following uniform estimates:
\begin{align}
\| v - \hat{B}^{ns}_k[A_kv] \|_{A_k}^2  &\leq \delta \| v \|_{A_k}^2,  \label{ine:ns_hatB_uniform} \\
\| v - \tilde{B}^{ns}_k[A_kv] \|_{A_k}^2  &\leq \delta^n \| v \|_{A_k}^2,   \label{ine:ns_tildeB_uniform}
\end{align}
where $\delta = \frac{c_1}{c_1+c_2} < 1$, which is a constant independent of~$k$.
\end{theorem}

\begin{proof}
\eqref{ine:ns_hatB_uniform}~and~\eqref{ine:ns_tildeB_uniform}~follow directly from the comparison Theorem \ref{thm:ns_compare}, Theorem \ref{thm:converge-NPCG}, and the uniform convergence results of the $\backslash$-cycle MG method under Assumptions~\ref{asmp:approx} and \ref{asmp:smooth}.
\end{proof}

In the next theorem, we study the uniform convergence of the symmetric nonlinear AMLI-cycle MG under  Assumptions~\ref{asmp:approx}~and~\ref{asmp:smooth}.
\begin{theorem} \label{thm:uniform}
Let~$\hat{B}_k[\cdot]$~be defined by Algorithm~\ref{alg:AMLI}, and let~$\tilde{B}_k[\cdot]$~be implemented as in Algorithm~\ref{alg:nonlinear-CG} with~$\hat{B}_k[\cdot]$~as the preconditioner. Assume that  Assumptions~\ref{asmp:approx}~and~\ref{asmp:smooth}~hold.  Then we have the following uniform estimates:
\begin{align}
\| v - \hat{B}_k[A_kv] \|_{A_k}^2  &\leq \delta \| v \|_{A_k}^2,  \label{ine:hatB_uniform} \\
\| v - \tilde{B}_k[A_kv] \|_{A_k}^2  &\leq \delta^n \| v \|_{A_k}^2,   \label{ine:tildeB_uniform}
\end{align}
where $\delta = \frac{c_1}{c_1+c_2} < 1$ is a constant independent on~$k$.
\end{theorem}
\begin{proof}
According to Theorem~\ref{thm:AMLI_vs_V}~and the uniform convergence results of the V-cycle MG method under Assumptions~\ref{asmp:approx} and \ref{asmp:smooth}, we have
$
(v - \hat{B}_k[A_kv], v)_{A_k} \leq (v - B_kA_kv, v)_{A_k} \leq \delta^{1/2} (v,v)_{A_k}.
$
Then~\eqref{ine:hatB_uniform}~follows directly.  In addition, \eqref{ine:tildeB_uniform}~follows from \eqref{ine:improve}, \eqref{ine:hatB_uniform} and Theorem~\ref{thm:converge-NPCG}.
\end{proof}

In Theorems~\ref{thm:ns_uniform}~and~\ref{thm:uniform}, the full version of the nonlinear PCG is (implicitly) assumed. However, it is clear that as we only use the minimization property~\eqref{ine:improve}~in the proof,  the final result also holds for any truncated version of the nonlinear PCG.
Therefore, we have the following two corollaries regarding the uniform convergence of the nonlinear AMLI-cycle MG method using truncated versions of the 
nonlinear PCG method.
\begin{corollary}
Let~$\hat{B}_k^{ns}[\cdot]$~be defined by Algorithm~\ref{alg:ns-AMLI}, and let~$\tilde{B}_k^{ns}[\cdot]$~be implemented in a truncated version of Algorithm~\ref{alg:nonlinear-CG} 
with~$\hat{B}_k^{ns}[\cdot]$~as the preconditioner. Assume that  Assumptions~\ref{asmp:approx}~and~\ref{asmp:smooth}~hold.  Then we have the following uniform estimates:
\begin{equation*}
\| v - \hat{B}^{ns}_k[A_kv] \|_{A_k}^2  \leq \delta \| v \|_{A_k}^2, \qquad
\| v - \tilde{B}^{ns}_k[A_kv] \|_{A_k}^2  \leq \delta^n \| v \|_{A_k}^2, 
\end{equation*}
where $\delta = \frac{c_1}{c_1+c_2} < 1$ is a constant independent on~$k$.
\end{corollary}

\begin{corollary}
Let~$\hat{B}_k[\cdot]$~be defined by Algorithm~\ref{alg:AMLI}, and let~$\tilde{B}_k[\cdot]$~be implemented in a truncated version of Algorithm~\ref{alg:nonlinear-CG} with $\hat{B}_k[\cdot]$ as the preconditioner. Assume that  Assumptions~\ref{asmp:approx} and \ref{asmp:smooth}~hold.  Then we have the following uniform estimates:
\begin{equation*}
\| v - \hat{B}_k[A_kv] \|_{A_k}^2  \leq \delta \| v \|_{A_k}^2,  \qquad
\| v - \tilde{B}_k[A_kv] \|_{A_k}^2  \leq \delta^n \| v \|_{A_k}^2, 
\end{equation*}
where $\delta = \frac{c_1}{c_1+c_2} < 1$ is a constant independent on~$k$.
\end{corollary}

\begin{remark}
In \cite{Notay.Y;Vassilevski.P2008}, the uniform convergence of the nonlinear AMLI-cycle MG (only for~$\tilde{B}_k[\cdot]$) is shown if the number of nonlinear PCG iterations is chosen to be sufficiently large (under a certain assumption on the boundedness of the V-cycle MG with bounded-level difference). However, this condition is not needed for above results.  Our uniform convergence results hold for arbitrary choices of the number of the nonlinear PCG iterations, but instead our results require Assumptions~\ref{asmp:approx}~and~\ref{asmp:smooth}.
\end{remark}  

\subsection{Uniform Convergence without Assumptions~\ref{asmp:approx}~and~\ref{asmp:smooth}}
So far, the convergence results suggest that not only $\tilde{B}_k[\cdot]$ but also $\hat{B}_k[\cdot]$ converges uniformly under Assumptions~\ref{asmp:approx}~and~\ref{asmp:smooth}. A natural question arises, under the same assumption on the bounded convergence factor 
 of the V-cycle MG with the bounded-level difference, $k_0$, used in \cite{Notay.Y;Vassilevski.P2008}, does the nonlinear operator $\hat{B}_k[\cdot]$ converge uniformly when $n$ is sufficiently large? The following two theorems each gives a positive answer to this question. This is a (slight) generalization of the result in \cite{Notay.Y;Vassilevski.P2008} with a simpler proof.

For the sake of simplicity, let us assume that the convergence factor of the two-grid method ($k_0 = 1$) is independent of $k$. The more general case, when the convergence factor of the V-cycle MG with bounded-level difference $k_0$ is independent of $k$, can be analyzed similarly.  

\begin{theorem}\label{thm:ns_uniform_n}
Let~$\hat{B}_k^{ns}[\cdot]$~be defined by Algorithm~\ref{alg:ns-AMLI}, and let~$\tilde{B}_k^{ns}[\cdot]$~be implemented as in Algorithm~\ref{alg:nonlinear-CG} with~$\hat{B}_k^{ns}[\cdot]$~as the preconditioner. Assume that the convergence factor of the two-grid method is bounded by $\bar{\delta} \in [0, 1)$ which is independent of $k$.  Let $n$, the number of iterations of the nonlinear PCG method, be chosen such that the inequality
\begin{equation}\label{ine:n}
(1-\delta^n)\bar{\delta} + \delta^n \leq \delta
\end{equation}
has a solution $\delta \in [0,1)$. A sufficient condition for this is
\begin{equation}\label{ine:n_large}
n > \frac{1}{1-\bar{\delta}}.
\end{equation}
Then we have the following uniform estimates:
\begin{align}
\| v - \hat{B}^{ns}_k[A_kv] \|_{A_k}^2  &\leq \delta \| v \|_{A_k}^2,  \label{ine:ns_hatB_uniform_n} \\
\| v - \tilde{B}^{ns}_k[A_kv] \|_{A_k}^2  &\leq \delta^n \| v \|_{A_k}^2,   \label{ine:ns_tildeB_uniform_n}
\end{align}
where $\delta$ is independent of $k$.
\end{theorem}

\begin{proof}
We prove the estimates by mathematical induction. The results hold for $k=1$ trivially. Assume that~\eqref{ine:ns_hatB_uniform_n} and \eqref{ine:ns_tildeB_uniform_n} hold for $k-1$, and let $\hat{v} = (I - R_kA_k)v$. Similar to Theorem~\ref{thm:ns_compare_n}, we have
\begin{equation*}
\|v - \hat{B}^{ns}_k[A_k v ] \|_{A_k}^2  \leq ((1-\delta^n) \bar{\delta} + \delta^n) \| v \|_{A_k}^2  \leq \delta \| v \|_{A_k}^2. 
\end{equation*}
This shows that estimate~\eqref{ine:ns_hatB_uniform_n}~holds. Moreover, according to Theorem~\ref{thm:converge-NPCG} and when $f = A_kv$ in~\eqref{ine:converge-NPCG}, the estimate~\eqref{ine:ns_tildeB_uniform_n}~follows directly. 

Now we show that~\eqref{ine:n_large} implies the existence of a $\delta$, which solves~\eqref{ine:n}. Solving \eqref{ine:n} is equivalently to solving
$
\phi(\delta) \equiv  (1+\delta + \delta^2 + \cdots + \delta^{n-1})\bar{\delta} - (\delta + \delta^2 + \cdots + \delta^{n-1}) \leq 0,
$
as $\phi(\delta)(1-\delta) = (1-\delta^n)\bar{\delta} + \delta^n - \delta$.  Due to~\eqref{ine:n_large}, $\phi(1) = n \bar{\delta} - (n-1) = 1 - n (1-\bar{\delta})
< 0$, and $\phi(0) = \bar{\delta} > 0$.  Therefore, there is a $\delta^* \in [0,1)$  such that $\phi(\delta^*) = 0$. 
Then any $\delta \in [ \delta^*, 1)$ will satisfy~\eqref{ine:n}.
\end{proof}

A similar result is obtained for symmetric case as the following theorem shows. 

\begin{theorem} \label{thm:uniform_n}
Let~$\hat{B}_k[\cdot]$~be defined by Algorithm~\ref{alg:AMLI}, and let~$\tilde{B}_k[\cdot]$~be implemented as in Algorithm~\ref{alg:nonlinear-CG} with~$\hat{B}_k[\cdot]$~as the preconditioner. Assume that the convergence factor of the two-grid method is bounded by $\bar{\delta} \in [0, 1)$ which is independent of $k$.  Let $n$, the number of iterations of the nonlinear PCG method, be chosen such that the inequality
\begin{equation*}
(1-\delta^n)\bar{\delta} + \delta^n \leq \delta.
\end{equation*}
has a solution $\delta \in [0,1)$.  A sufficient condition for this is
\begin{equation*}
n > \frac{1}{1-\bar{\delta}}.
\end{equation*}
 Then we have the uniform estimates:
\begin{align}
\| v - \hat{B}_k[A_kv] \|_{A_k}^2  &\leq \delta \| v \|_{A_k}^2,  \label{ine:hatB_uniform_n} \\
\| v - \tilde{B}_k[A_kv] \|_{A_k}^2  &\leq \delta^n \| v \|_{A_k}^2,   \label{ine:tildeB_uniform_n}
\end{align}
where $\delta$ is independent of $k$.
\end{theorem}

\begin{proof}
The results hold for $k=1$ trivially.  Assume that~\eqref{ine:hatB_uniform_n}~and~\eqref{ine:tildeB_uniform_n}~hold for level~$k-1$. Similar to Theorem~\ref{thm:compare_n}, we have
$$
( v - \hat{B}_k[A_kv], w)_{A_k} \leq \sqrt{\| \hat{v} - P_{k-1} \hat{v} \|_{A_k}^2 + \delta^n \| P_{k-1}\hat{v}  \|_{A_k}^2} \times \| \hat{w} \|_{A_k},
$$
where~$\hat{v} = (I-R_kA_k)v$~and~$\hat{w} = (I-R_kA_k)w$.  The first term on the right hand side can be estimated by the same argument as in Theorem~\ref{thm:ns_uniform_n}.  Therefore, we have
$
( v - \hat{B}_k[A_kv], w)_{A_k} \leq \delta^{1/2} \| v\|_{A_k} \| w \|_{A_k}, 
$
which implies~\eqref{ine:hatB_uniform_n}.  According to Theorem~\ref{thm:converge-NPCG}, the estimate~\eqref{ine:tildeB_uniform_n}~follows directly.  The existence of $\delta$ has already been shown in Theorem~\ref{thm:ns_uniform_n}. 
\end{proof}

For $k_{0}=1$ and $n =2$, the nonlinear AMLI-cycle MG has the complexity of the W-cycle MG method and the sufficient condition~\eqref{ine:n_large}~becomes
\begin{equation*}
2 = n > \frac{1}{1-\bar{\delta}} \Rightarrow \bar{\delta} < \frac{1}{2}.
\end{equation*}
In conclusion, we have the following result.
\begin{corollary}
If the two-grid method at any level $k$ (with an exact solution at the coarse level $k+1$) has a uniformly bounded convergence rate $\bar{\delta} < \frac{1}{2}$, then the respective nonlinear AMLI-cycle MG with $n=2$ converges uniformly. 
\end{corollary}

\begin{remark}
As Theorem~\ref{thm:converge-NPCG}~holds for both the full and the truncated version of the nonlinear PCG methods, the above uniform convergence estimates hold for both of these methods likewise.
\end{remark}

\section{Numerical Experiments}
In this section, we present some numerical results to illustrate our theoretical results.    
The first model problem we consider here is 
\begin{equation} \label{eqn:pde_model}
\begin{cases}
-\Delta u &= f,  \quad \text{in} \ \Omega,    \\
u &=  0,  \quad  \text{on} \ \partial \Omega,
\end{cases}
\end{equation}
where $\Omega$ is the unit square in $\mathbb{R}^2$. In our numerical experiments, we discretize equation~\eqref{eqn:model}~using the linear 
finite element method with the choice of $f=1$. The domain $\Omega$ is triangulated by uniform refinements, and the mesh size on the finest level is $h = 2^{-k}$ where $k$ is the number of levels used.

In Table~\ref{tab:AMLI_vs_V}, the numerical results of the nonlinear AMLI-cycle MG method and the V-cycle MG methods are presented and compared.  Under the setting of our experiments, Assumptions~\ref{asmp:approx}~and~\ref{asmp:smooth} are satisfied.  Then according to Theorem~\ref{thm:uniform}, both of the nonlinear operators $\hat{B}_k[\cdot]$ and $\tilde{B}_k[\cdot]$ are uniformly convergent, as illustrated by the numerical results shown in Table~\ref{tab:AMLI_vs_V}. Furthermore, $\hat{B}_k$ and $\tilde{B}_k$ are better than $B_k$ in terms of the number of iterations, which agrees with 
Theorem~\ref{thm:AMLI_vs_V}. 
\begin{table}[!htbp]
 \centering 
\caption{\label{tab:AMLI_vs_V} Number of iterations of the V-cycle MG and the nonlinear AMLI-cycle MG. (stopping criteria: relative residual is less than $10^{-6}$; N-PCG($n$): $n$ iterations of the nonlinear PCG is used to define the coarse-level solver $\tilde{B}_{k-1}[\cdot]$)}
\begin{tabular}{|c||c|c|c|c|c|} 
\hline
         &$B_k$ & \multicolumn{2}{c|}{$\hat{B}_k[\cdot]$}&\multicolumn{2}{c|}{$\tilde{B}_k[\cdot]$}  \\\hline \hline
 $k$ &            &   N-PCG(1)  &  N-PCG(2) & N-PCG(1) & N-PCG(2) \\ \hline
 5     &9          &9		  &9         	       &7		  &3             \\
 6     &11       &10		  &10       	       &8		  &4     \\
 7     &12       &11		  &10       	       &9		  &4         \\
 8     &13       &11		  &10                &10		  &4         \\
 9     &13       &12		  &10                &10		  &4         \\
 10  &14       &12		  &10                &11		  &4       \\
 11  &14       &12		  &10                &12		  &4       \\
 12  &14       &13		  &10                &12  		  &4         \\
        \hline
\end{tabular}
\end{table}

The second model problem is a diffusion equation with a large jump in the coefficient: 
 \begin{equation*} 
 \begin{cases}
-\nabla \cdot ( a(x) \nabla u) &= f,  \quad \text{in} \ \Omega,    \\
u &=  0,  \quad  \text{on} \ \partial \Omega,
\end{cases}
\end{equation*}
where $\Omega = (0,1) \times (0,1)$. We have $a(x) = 1$ on $\Omega_1 = (0.25, 0.5) \times (0.25, 0.5)$ and $\Omega_2 = (0.5, 0.75) \times (0.5, 0.75)$, as well as $a(x) = 10^{-6}$ on $\Omega \backslash (\bar{\Omega}_1 \cup \bar{\Omega}_2  )$. The domain $\Omega$ is triangulated by uniform refinements, and the mesh size on the finest level is $h = 2^{-k}$, where $k$ is the number of levels used. In this test problem, we choose $f = 0$, which means the exact solution is $u^* = 0$.  As we know the exact solution, the stopping criteria is $|| u^* - u_i ||_{A} \leq 10^{-6}$ where $u_i$ is the $i$-th iteration of the MG method.

It is well known that the performance of the V-cycle MG methods for this jump coefficient problem will degenerate. Table~\ref{tab:AMLI_vs_V_jump}~confirms this fact. 
For this problem, due to lack of regularity, if one iteration of the nonlinear PCG method is used to define the coarse-level solver, both $\hat{B}_k[\cdot]$ and $\tilde{B}_k[\cdot]$ appear to be nonuniformly convergent. Nevertheless, according to Theorem~\ref{thm:AMLI_vs_V}, both $\hat{B}_k[\cdot]$ and $\tilde{B}_k[\cdot]$ exhibit better convergence 
than does the V-cycle MG. Furthermore, if the number of iterations of the nonlinear PCG methods is sufficiently large ($n=2$ in this case), according to the theoretical results in \cite{Notay.Y;Vassilevski.P2008}, we can expect that the $\tilde{B}_k[\cdot]$ to be uniformly convergent both with respect to the number of levels $k$ and the jumps, as shown by the numerical results in Table~\ref{tab:AMLI_vs_V_jump}. Furthermore, we see that $\hat{B}_k[\cdot]$ also converges uniformly. 

\begin{table}[!htbp]
 \centering 
\caption{\label{tab:AMLI_vs_V_jump} Number of iterations of the V-cycle MG and the nonlinear AMLI-cycle MG for the jump coefficient problem. (stopping criteria: energy norm of error is less than $10^{-6}$; N-PCG($n$): $n$ iterations of the nonlinear PCG is used to define the coarse level solver $\tilde{B}_{k-1}[\cdot]$)}
\begin{tabular}{|c||c|c|c|c|c|} 
\hline
         &$B_k$ & \multicolumn{2}{c|}{$\hat{B}_k[\cdot]$}&\multicolumn{2}{c|}{$\tilde{B}_k[\cdot]$}  \\\hline \hline
 $k$ &            &   N-PCG(1)  &  N-PCG(2) & N-PCG(1) & N-PCG(2) \\ \hline
 5   	&27       &15		  &13         	       &13		  &4             \\
 6      &40       &22		  &14       	       &15		  &5    	 \\
 7   	&49       &29		  &14       	       &20		  &5    	     \\
 8   	&56       &37		  &15       	       &30		  &5    	     \\
 9   	&76       &45		  &15       	       &42		  &5    	     \\
 10  	&102     &55  		  &15      	       &47		  &5    	   \\
        \hline
\end{tabular}
\end{table}

In the last numerical experiment, we use the unsmoothed aggregation AMG (UA-AMG) methods 
to solve the model problem~\eqref{eqn:pde_model} discretized by the linear finite element on uniform meshes. 
Given the $k$-th level matrix $A_k \in \mathbb{R}^{n_k \times n_k}$, in the UA-AMG method we define the prolongation matrix $P_{k-1}^{k}$ from a non-overlapping partition of the $n_k$ unknowns at level $k$ into $n_{k-1}$ nonempty disjoint sets $G_j$, $j=1,\dots, n_{k-1}$, which are referred to as  aggregates. In our numerical experiments, we use Algorithm 2 in \cite{Vanek.P;Mandel.J;Brezina.M1996} to generate the aggregates on each level. Once the aggregates are constructed, the prolongation $P_{k-1}^{k}$ is an $n_{k} \times n_{k-1}$ matrix given by 
\begin{equation*}
(P_{k-1}^{k})_{ij} = 
\begin{cases}
1 & \text{if} \ i \in G_j  \\
0 & \text{otherwise}
\end{cases} \quad
i = 1, \dots, n_k, \quad j = 1, \dots, n_{k-1}. 
\end{equation*}
With such piecewise constant prolongation $P_{k-1}^{k}$, the coarse-level matrix $A_{k-1} \in \mathbb{R}^{n_{k-1} \times n_{k-1}}$ is defined by
$
A_{k-1} = (P_{k-1}^{k})^t A_k (P_{k-1}^{k}).
$ 
As we are now considering AMG methods, we do not have the orthogonal projections $Q_k$ and $P_k$, therefore, we cannot use these projections to define the operators $B_k$, $\hat{B}_k[\cdot]$, and $\tilde{B}_k[\cdot]$. However, thanks to the prolongation $P_{k-1}^{k}$, the V-cycle MG iterator $B_k$ for the UA-AMG method is defined recursively by Algorithm~\ref{alg:V-cycle} with the coarse-grid correction step
$
u_2 = u_1 + P_{k-1}^{k} B_{k-1} (P_{k-1}^{k})^t (f - A_k u_1). 
$ 
Similarly, for UA-AMG method, the nonlinear operator $\hat{B}_k[\cdot]$ is defined by Algorithm~\ref{alg:AMLI} with the coarse-grid correction step
$
u_2 = u_1 +  P_{k-1}^{k} \tilde{B}_{k-1}[(P_{k-1}^{k})^t(f-A_k u_1)]
$ 
and the nonlinear operator $\tilde{B}_k[\cdot]$ is implemented as in Algorithm~\ref{alg:nonlinear-CG} with $\hat{B}_k[\cdot]$ as the preconditioner.

For the model problem~\eqref{eqn:pde_model} discretized by the linear finite element method, it can be shown that the two-grid UA-AMG method converges uniformly.  When what is known as the XZ-identity~\cite{Xu.J;Zikatanov.L2002}~is applied to the two-grid method, the following well-known result is obtained (see, e.g. \cite{Zikatanov.L.2008a,Falgout.R;Vassilevski.P;Zikatanov.L.2005a})
\begin{equation}\label{TG identity}
\| I - B_{TG} A_{k} \|_{A_k}^2 = 1 - \frac{1}{c^2},
\end{equation}
where 
\begin{equation}\label{TG constant}
c^2 = \sup_{\| v\|_{A_k}=1} \inf_{v_{k-1} \in V_{k-1}}  (\bar{R}_k^{-1} (v - v_{k-1}), v- v_{k-1}).  
\end{equation}
For the UA-AMG method, due to the piecewise constant prolongation $P_{k-1}^k$, the entries of the coarse-level matrix $A_k$ behave like $\mathcal{O}((\frac{h_k}{h}))$ instead of the usual $\mathcal{O}(1)$.  Here, $h_k$ is the $k$-th level mesh size, i.e., the diameter of the aggregates, and $h$ is the mesh size of the finest level.  Therefore, we have $\lambda_k:=\rho(A_k)  = \mathcal{O}((\frac{h_k}{h})h_k^{-2})$.  In considering any smoother that is equivalent to the Richardson smoother, we have the smoothing property
\begin{equation*}
(\bar{R}_k^{-1}v,v) \leq c_0 \lambda_k (v,v) \leq c_1 \frac{h_k}{h} h_k^{-2} (v,v).
\end{equation*}
Moreover, by choosing $v_{k-1}$ as the piecewise constant interpolation of $v$ on each aggregate, we obtain the approximation property 
\begin{equation*}
\| v - v_{k-1} \|^2 \leq c_2 h_{k-1}^2 | v |^2_1 \leq c_3 h_{k-1}^2 \frac{h}{h_k} \| v \|^{2}_{A_k},
\end{equation*}
where $| v |_1$ is the standard semi-norm of the Sobolev space $H^1(\Omega)$.  Therefore, the constant $c$ in the two-grid convergence estimates \eqref{TG identity} - \eqref{TG constant} can be estimated by
\begin{align*}
(\bar{R}_k^{-1}(v-v_{k-1}), v-v_{k-1}) \leq c_1\frac{h_k}{h} h_{k}^{-2}(v-v_{k-1}, v-v_{k-1}) \leq c_1c_3 h_k^{-2} h_{k-1}^2 
\| v\|^{2}_{A_k}.
\end{align*}
It is reasonable to assume that between two successive levels the aggregates are quasi-uniform.  Hence $h_{k-1}/h_k \leq c_4$ where $c_4$ is a constant independent on the levels.  Aggregation algorithms that can ensure quasi-uniformity can be found, for example, in \cite{BrezinaVanekVassilevski2012}.

This implies that $c \leq c_1c_3c_4^2$ and does not depend on the mesh size, which shows the uniform convergence of two-grid UA-AMG method.  According to Theorems~\ref{thm:ns_uniform_n} and \ref{thm:uniform_n}, we can expect that when $n$ is sufficiently large, the UA-AMG method with the nonlinear AMLI-cycle MG method converges uniformly.  In fact, for model problem~\eqref{eqn:pde_model}, $n=2$~is sufficient and gives uniformly convergent results. 

The results are shown in Table~\ref{tab:AMLI_vs_V_UA}. We set the maximum size of the aggregates to be $9$.  We can see that if we use the V-cycle MG for UA-AMG, the number of iterations depends strongly on the size of the problem. 
If we use one iteration of the nonlinear PCG to define the coarse level solver, the performance of $\hat{B}_k[\cdot]$ and of $\tilde{B}_k[\cdot]$ each still depends on the size of the problem, but the number of iterations grows considerably less quickly. If we use $n=2$ iterations, $\hat{B}_k[\cdot]$ and $\tilde{B}_k[\cdot]$ each exhibits uniform convergence. 

\begin{table}[!htbp]
 \centering 
\caption{\label{tab:AMLI_vs_V_UA} Number of iterations of the V-cycle MG and the nonlinear AMLI-cycle MG for the UA-AMG method. (stopping criteria: relative residual is less than $10^{-6}$; N-PCG($n$): $n$ iterations of the nonlinear PCG is used to define the coarse-level solver $\tilde{B}_{k-1}[\cdot]$)}
\begin{tabular}{|c||c|c|c|c|c|} 
\hline
         &$B_k$ & \multicolumn{2}{c|}{$\hat{B}_k[\cdot]$}&\multicolumn{2}{c|}{$\tilde{B}_k[\cdot]$}  \\\hline \hline
 Size     	&            &   N-PCG(1)  &  N-PCG(2) & N-PCG(1) & N-PCG(2) \\ \hline
 3,969    	&100       &48		  &40         	       &34		  &9             \\
 16,129  	&244       &70		  &41       	       &38		  &9    	 \\
 65,025  	&519       &94		  &41       	       &56		  &9    	     \\
 261,121	&713       &93		  &41       	       &63		  &9    	     \\
 1,046,529&1753    &112		  &40      	       &93		  &9    	     \\
        \hline
\end{tabular}
\end{table}

The last experiments demonstrate the potential of the nonlinear AMLI-cycle MG methods in cases where the constructed hierarchy of the interpolation matrices 
is not energy stable. In many cases, it is straightforward to come up with simple (e.g., block--diagonal) interpolation matrices.  However, these lead to a V-cycle MG that generally has level--dependent convergence. The nonlinear AMLI-cycle can  
be used in such instances to substantially improve convergence (cf., e.g., \cite{Lashuk.I;Vassilevski.P2008}). 

\section*{Acknowledgement} We would like to thank Professor Ludmil Zikatanov at Penn State University for generously sharing valuable insights that we drew on in writing this paper. The valuable and careful comments offered by anonymous referees are also gratefully acknowledged.

\bibliographystyle{siam}
\bibliography{NonlinearAMLI}

\end{document}